\documentclass{amsart}

\title[ Vaisman metrics on  solvmanifolds]
{Vaisman metrics on solvmanifolds and  Oeljeklaus-Toma manifolds}

\author{Hisashi Kasuya}
\usepackage{amssymb}
\usepackage{amsmath}
\usepackage{amscd}
\usepackage{amstext}
\usepackage{amsfonts}
\usepackage[all]{xy}

\newcommand{\C}{\mathbb{C}}
\newcommand{\R}{\mathbb{R}}

\newcommand{\Q}{\mathbb{Q}}
\newcommand{\Z}{\mathbb{Z}}
\newcommand{\g}{\frak{g}}

\newtheorem{theorem}{Theorem}[section]
\newtheorem{lemma}[theorem]{Lemma}    
\newtheorem{corollary}[theorem]{Corollary}
\newtheorem{proposition}[theorem]{Proposition}

\newtheorem{definition}{Definition}
\newtheorem{remark}{Remark}
\newtheorem{problem}{Problem}
\newtheorem{example}{Example}
\begin{document} 

\maketitle
\begin{abstract}
We prove the non-existence of Vaisman metrics on some  solvmanifolds with  left-invariant complex structures.
By  this theorem, we show that  Oeljeklaus-Toma manifolds does not admit  Vaisman metrics.
\end{abstract}
\section{Introduction}
A Hermitian metric $g$ on a complex manifold is locally conformal K\"ahler (LCK) if there exists a closed $1$-form $\theta$ (called the Lee from)  such that $d\omega=\theta\wedge \omega$ where $\omega$ is the fundamental form of $g$.
A LCK metric $g$ is  Vaisman if  the Lee form $\theta$ is parallel.
It is known that Vaisman manifolds have some special properties not shared by LCK manifolds.
For example \cite{V},\cite{Kas} showed that the first Betti number $b_{1}$ of a Vaisman manifold is odd,
whereas an LCK manifold with even $b_{1}$ is presented in \cite{OT}.

Let $G$ be a simply connected solvable Lie group with a lattice (i.e. cocompact discrete subgroup) $\Gamma$.
We call $G/\Gamma$ a solvmanifold.
If $G$ is nilpotent, then we call $G/\Gamma$ a nilmanifold.
We are interested in studying LCK and Vaisman structures on solvmanifolds.
Suppose $G$ is nilpotent and $G$ admits a left-invariant complex structure $J$.
It is known that the nilmanifold $(G/\Gamma,J)$ admits a LCK metric if and only if $G=\R\times H(n)$ where $H(n)$ is the $(2n+1)$-dimensional Heisenberg Lie group (see \cite{Saw}).
On the other hand, not much is known about LCK and Vaisman structures on general solvmanifolds.
The purpose of this paper is to prove non-existense of Vaisman metrics on some  solvmanifolds with left-invariant complex structures.
We prove:
\begin{theorem}\label{MTTTTT}
Let $G=\R^{m}\ltimes_{\phi} \R^{n}$  such that $\phi$ is a semi-simple action.
Suppose  $\dim [G,G]>\frac{\dim G}{2}$, $G$ has a lattice $\Gamma$ and a left-invariant complex structure $J$ and $b_{1}(G/\Gamma)=b_{1}(\g)$.
Then $(G/\Gamma,J)$ admits no Vaisman metric. 
\end{theorem}
Since $\R\times H(n)$ admits a left-invariant Vaisman metric, a nilmanifold with a left-invariant complex structure admits a LCK metric if and only if it admits a Vaisman metric.
But for solvmanifolds, by Theorem \ref{MTTTTT}, it is to be expected that we obtain many non-Vaisman LCK manifolds.

We call a solvmanifold $G/\Gamma$  meta-abelian  if $G=\R^{m}\ltimes_{\phi} \R^{n}$  such that $\phi$ is a semi-simple action.
On some   meta-abelian solvmanifolds, we can find various non-K\"ahler complex geometric structures.
For example,  pseudo-K\"ahler structures (see \cite{Y}) and generalized K\"ahler structures (see \cite{FT}).
In \cite{K1} the author showed that  meta-abelian  solvmanifolds are formal in the sense of Sullivan (moreover geometrically formal in the sense of  Kotschick \cite{KF}) and satisfy  hard Lefschetz property if they admit symplectic structures.
We note that a solvmanifold admits a K\"ahler metric  if and only if it is a finite quotient of a complex torus which has a structure of a complex torus bundle over a complex torus  (see \cite{Hn} and \cite{BC}).

In Theorem \ref{MTTTTT}, the assumption  $b_{1}(G/\Gamma)=b_{1}(\g)$ is important.
In this paper we study a criterion for the condition  $H^{1}(\g)\cong H^{1}(G/\Gamma)$.
We prove:
\begin{theorem}
For a weakly completely solvable Lie group $G$ with a lattice $\Gamma$, we have an isomorphism
\[H^{1}(\g)\cong H^{1}(G/\Gamma).
\]
\end{theorem}
 A weakly completely solvable Lie group (see Definition \ref{wea}) is more general than a completely solvable Lie group.
Thus this theorem is a generalization of Hattori's theorem \cite{Hatt} for the first cohomology.

Important examples which we can apply Theorem \ref{MTTTTT} to are Oeljeklaus-Toma(OT) manifolds.
In \cite{OT}, for any integers  $s>0$ and $t>0$, Oeljeklaus and Toma constructed  compact complex manifolds (OT-manifolds of type $(s,t)$) with  Betti number $b_{1}=s$  by using number theory.
Oeljeklaus and Toma showed that for any integers $s>0$ OT-manifolds of type $(s,1)$ admit LCK metrics.
For even $s>0$, we can say that these  admit no  Vaisman metric (moreover  OT-manifolds of type $(2,1)$ are counter examples to Vaisman's conjecture).
But for odd $s>0$, it was not known whether   OT-manifolds of type $(s,1)$ admit   Vaisman metrics.
In this paper we represent OT-manifolds as solvmanifolds and we prove:
\begin{corollary}
OT-manifolds do not   admit  Vaisman metrics.
\end{corollary}

\section{$d_{\theta}$-cohomology}
Let $M$ be a manifold and $A^{\ast}(M)$ the de Rham complex of $M$ with the exterior differential $d$.
For a closed $1$-form $\theta\in A^{\ast}(M)$  we define the new differential $d_{\theta}:A^{p}(M)\to A^{p+1}(M)$ by $d_{\theta}(\alpha)=d\alpha-\theta\wedge \alpha$.
We denote by $H^{\ast}_{\theta}(M)$ the $d_{\theta}$-cohomology and by $[\alpha]_{\theta}$ the $d_{\theta}$-cohomology class of a $d_{\theta}$-closed form $\alpha$.

Let $G$ be a simply connected solvable Lie group with a lattice $\Gamma$ and $\g$ be the Lie algebra of $G$.
Consider the exterior algebra $\bigwedge \g^{\ast} $ of the dual space of $\g$.
Let $d:\bigwedge^{1}\g\to \bigwedge^{2}\g$ be the dual map of the Lie bracket of $\g$ and    $d:\bigwedge^{p}\g\to \bigwedge^{p+1}\g$ the extension of this map.
We can identify  $(\bigwedge \g^{\ast}, d)$ with the left-invariant forms on $G$ with the exterior differential.
By the invariant condition, we also consider $(\bigwedge \g^{\ast}, d)$ as the subcomplex of $A^{\ast}(G/\Gamma)$.
Let $\theta\in\bigwedge \g^{\ast}$ be a closed left-invariant $1$-form. 
We denote by $H^{\ast}_{\theta}(\g)$ the cohomology of the cochain complex $\bigwedge \g^{\ast}$ with the differential $d_{\theta}$.
A simply connected solvable Lie group with a lattice is unimodular (see \cite[Remark 1.9]{R}).
Let $d\mu$ be a bi-invariant volume form such that $\int_{G/\Gamma}d\mu =1$.
For $\alpha\in A^{p}(G/\Gamma)$, we have a left-invariant form $\alpha_{inv}\in \bigwedge^{p} \g^{\ast}$ defined  by
\[\alpha_{inv}(X_{1},\dots ,X_{p})=\int_{G/\Gamma}\alpha(\tilde X_{1},\dots ,\tilde X_{p})d\mu
\]
for $X_{1},\dots ,X_{p}\in \g$ where $\tilde X_{1},\dots ,\tilde X_{p}$ are vector fields on $G/\Gamma$ induced by $X_{1},\dots X_{p}$.
We define the map $I:A^{\ast}(M)\to  \bigwedge \g^{\ast}$ by $\alpha \mapsto \alpha_{inv}$.
\begin{lemma}\label{sekib}
For any closed left-invariant  $1$-form $\theta$,  $I:(A^{\ast}(G/\Gamma),d_{\theta})\to (\bigwedge \g^{\ast},d_{\theta})$ is a homomorphism of cochain complexes and satisfies $I\circ i=id_{\bigwedge \g^{\ast}}$ where $i:\bigwedge \g^{\ast}\to A^{\ast}(G/\Gamma)$ is the above inclusion.
Hence the induced map $i^{\ast}:H^{\ast}_{\theta}(\g)\to H^{\ast}_{\theta}(G/\Gamma)$ is injective.
\end{lemma}
\begin{proof}
Consider 
\begin{multline*}
(d\alpha)_{inv}(X_{1},\dots ,X_{p+1})=\sum \int _{G/\Gamma} (-1)^{i+1}\tilde X_{i} (\alpha (\tilde X_{1},\dots ,{\hat {\tilde X}}_{i},\dots \tilde X_{p+1}))\\
+\sum(-1)^{i+j}\int_{G/\Gamma}\alpha([\tilde X_{i},\tilde X_{j}],\tilde X_{1},\dots ,{\hat {\tilde X}}_{i},\dots ,{\hat  {\tilde X}}_{j},\dots \tilde X_{p+1}).
\end{multline*}
In the proof of \cite[Theorem 7]{Bel}, it is proved that $\int_{G/\Gamma}A(F)d\mu=0$ 
for any function $F$ on $G/\Gamma$ and a left-invariant vector field $A$.
Thus $(d\alpha)_{inv}(X_{1},\dots, X_{p+1})=d(\alpha_{inv})(X_{1},\dots,X_{p+1})$.
Since $\theta$ is left-invariant, we have $(\theta \wedge \alpha)_{inv}(X_{1},\dots, X_{p+1})=\theta \wedge \alpha_{inv}(X_{1},\dots, X_{p+1})$.
Thus $I:(A^{\ast}(G/\Gamma),d_{\theta})\to (\bigwedge \g^{\ast},d_{\theta})$ is a homomorphism of cochain complexes.
Obviously we have $I\circ i=id_{\bigwedge \g^{\ast}}$.
\end{proof}

\section{LCK and Vaisman metrics}
Let $(M,J)$ be a complex manifold with a Hermitian metric $g$. 
We consider the fundamental form  $\omega=g(-,J-)$ of $g$.
The metric $g$ is  locally conformal K\"ahler (LCK)  if we have a closed $1$-form $\theta$ (called   the Lee form) such that $d\omega=\theta\wedge \omega$.
To study only non-K\"ahler LCK metrics, in this paper we assume  $\theta\not=0$ and $\theta$ is non-exact. 
Let $\nabla$ be the Levi-Civita connection of $g$.
A LCK metric $g$ is a Vaisman metric if  $ \nabla \theta=0$.
For a LCK manifold $(M,J,g)$ with  Lee form $\theta$,
the equation $d\omega=\theta\wedge \omega$ implies $d_{\theta}\omega =0$.
\begin{theorem}\label{OOR}{\rm (\cite{LL}) }
Let $(M,J)$ be a compact complex manifold admitting a Vaisman metric with the fundamental form $\omega$ and  Lee form $\theta$.
Then  the cohomology $H^{\ast}_{\theta}(M)$ is trivial.
In particular, we have $[\omega]_{\theta}=0$.
\end{theorem}

\begin{remark}
In addition to this theorem, if $(M,J)$ admits another LCK (not necessarily Vaisman) form $\omega_{0}$ with  Lee form $\theta_{0}$, then $\theta_{0}$ is cohomologous to $\theta$ and $[\omega_{0}]_{\theta_{0}}=0$ (see \cite{OV}).
\end{remark}
In this paper we also consider locally conformal symplectic (LCS) forms on $2n$-dimensional real manifolds.
They are non-degenerate $2$-forms $\omega$ such that there exists a closed $1$-form $\theta$ (also called the  Lee form) satisfying $d\omega=\theta\wedge \omega$.

For a Lie group $G$ we call a LCS form $\omega$ with  Lee form $\theta$ on $G$ (or $G/\Gamma$ if $G$ has a lattice $\Gamma$) a left-invariant LCS form if  $\omega\in \bigwedge \g^{\ast}$ and $\theta \in \bigwedge \g^{\ast}$.
Suppose $G$ admits a left-invariant complex structure $J$.
We call a Hermitian metric $g$ on $(G,J)$ (or $(G/\Gamma,J)$) a left-invariant LCK if $g$ is a left-invariant Hermitian metric and the fundamental form $\omega$ of $g$ is a left-invariant LCS form.

\section{Vaisman metrics on solvmanifolds}
First we prove:
\begin{lemma}
Let $G=\R^{m}\ltimes_{\phi}\R^{n}$  such that $\phi$ is a semi-simple action.
Then we can rewrite $G=\R^{m^{\prime}}\ltimes_{\psi} \R^{n^{\prime}}$ such that $\R^{n^{\prime}}$ has no trivial $\R^{m^{\prime}}$-submodule and $n^{\prime}=\dim [G,G]$.
\end{lemma}
\begin{proof}
Since $\phi$ is a semi-simple action,  we consider a decomposition $\R^{n}=V_{1}\oplus V_{2}$ such that $V_{1}$ is a maximal trivial $\R^{m}$-submodule and $V_{2}$ is its complement.
Then we have $G=V_{1}\times (\R^{m}\ltimes V_{2})$.
We notice that $\phi(\R^{m})$ is $\C$-diagonalizable.
Since we have $[G,G]=\{\phi(a)B-B\vert a\in \R^{m}, B\in V_{2} \}$ and $V_{2}$ has no trivial submodule, we have $\dim [G,G]= \dim V_{2}$.
\end{proof}
To prove Theorem \ref{MTTTTT} we prove:
\begin{theorem}\label{LCS}
Let $G=\R^{m}\ltimes_{\phi} \R^{n}$ such that $\phi$ is a semi-simple action.
Suppose  $\dim [G,G]>\frac{\dim G}{2}$ and $G$ has a lattice $\Gamma$.
Then for any left-invariant LCS form $\omega$ with  Lee form $\theta$, the $d_{\theta}$-cohomology class of $\omega $ is not $0$ in $H_{\theta}^{2}(G/\Gamma)$.
\end{theorem}
\begin{proof}
By the above lemma, we can assume $\R^{n}$ has no trivial $\R^{m}$-submodule and $n=\dim [G,G]$ and by $\dim [G,G]>\frac{\dim G}{2}$ we have $m<n$.
Consider the Lie algebra $\g={\frak a}\ltimes {\frak n}$ where $\frak a$ and $\frak n$ are abelian Lie algebras  corresponding to $\R^{m}$ and $\R^{n}$ respectively.
Then we have $\bigwedge \g^{\ast}=\bigwedge {\frak a}^{\ast}\otimes \bigwedge {\frak n}^{\ast}$.
Let $\omega\in \bigwedge \g^{\ast}$ be a non-degenerate left-invariant $2$-form.
For the direct sum $\bigwedge^{2} {\frak g}^{\ast}=\bigwedge^{2} {\frak a}^{\ast}\oplus (\bigwedge^{1} {\frak a}^{\ast}\otimes \bigwedge^{1} {\frak n}^{\ast})\oplus  \bigwedge^{2} {\frak n}^{\ast} $ consider the decomposition $\omega=\omega^{\prime}+\omega^{\prime\prime}$ such that $\omega^{\prime}\in \bigwedge^{2} {\frak a}^{\ast}\oplus (\bigwedge^{1} {\frak a}^{\ast}\otimes \bigwedge^{1} {\frak n}^{\ast})$ and $\omega^{\prime\prime}\in \bigwedge^{2} {\frak n}^{\ast} $.
Suppose $\omega^{\prime\prime}=0$.
Then we have
 \[\omega^{\frac {n+m}{2}}\in \bigoplus_{p>\frac {n+m}{2}, p+q=n+m} (\bigwedge^{p} {\frak a}^{\ast}\otimes \bigwedge^{q} {\frak n}^{\ast}).\]
By the assumption $m<n$, we have $\omega^{\frac {n+m}{2}}=0$, but this contradicts non-degeneracy of $\omega$.
Thus we have $\omega^{\prime\prime}\not=0$.
Assume $\omega$ is LCS and its Lee form is $\theta$.
Since we assume that ${\frak n}$ has no trivial $\frak a$-submodule, we have $[{\frak a}, {\frak n}]={\frak n}$.
This implies that ${\rm Ker}\, d_{\bigwedge^{1}\g^{\ast}}=\bigwedge^{1}{\frak a}^{\ast}$ and hence 
$\theta\in \bigwedge^{1}{\frak a}^{\ast}$.
By the semi-direct product $\g={\frak a}\ltimes {\frak n}$, we have
\[d(\bigwedge^{p} {\frak a}^{\ast}\otimes \bigwedge^{q} {\frak n}^{\ast})\subset \bigwedge^{p+1} {\frak a}^{\ast}\otimes \bigwedge^{q} {\frak n}^{\ast}.
\]
Then we have 
\[d_{\theta}(\bigwedge^{p} {\frak a}^{\ast}\otimes \bigwedge^{q} {\frak n}^{\ast}) \subset \bigwedge^{p+1} {\frak a}^{\ast}\otimes \bigwedge^{q} {\frak n}^{\ast},\]
and hence we have
\[d_{\theta}(\bigwedge^{1}\g) \subset \left( \bigwedge^{2} {\frak a}^{\ast}\otimes \bigwedge^{0} {\frak n}^{\ast}\right)\oplus \left(\bigwedge^{1} {\frak a}^{\ast}\otimes \bigwedge^{1} {\frak n}^{\ast} \right)
.\]
On the other hand, we have $\omega^{\prime\prime}\not=0$.
Hence $\omega=\omega^{\prime}+\omega^{\prime\prime}$ is not $d_{\theta}$-exact.
By Lemma \ref{sekib} this implies the theorem.
\end{proof}
\begin{remark}
The cohomology class of any  symplectic form $\omega$ on  a compact $2n$-dimensional manifold is non-trivial (moreover it satisfies $[\omega]^{n}\not=0$,  and conversely for  a solvmanifold $M$ a cohomology class $\Omega\in H^{2}(M)$ satisfying $\Omega^{n}\not=0$ contains a symplectic form see \cite{KS}).
But for a LCS form $\omega$ with  Lee form $\theta$, it is possible that  $[\omega ]_{\theta}=0$.
For examples consider a nilmanifold $G/\Gamma$. For a left-invariant LCS form $\omega$ on $G/\Gamma$ with the Lee form $\theta$, we have $[\omega]_{\theta}=0$ in $H^{\ast}_{\theta}(G/\Gamma)$ because for a non-zero closed left-invariant $1$-form $\theta$ the cohomology $H^{2}_{\theta}(\g^{\ast})$ is trivial (see \cite{Dix}).
Thus Theorem \ref{LCS} is a peculiar phenomenon on a  solvmanifold.
\end{remark}
\begin{example}{\bf (Another example)}
Consider $G=\R\ltimes _{\phi} H(1)$ such that $\phi$ is given by
\[\phi(t)\left(
\begin{array}{ccc}
1&  x&  z  \\
0&     1&y      \\
0& 0&1 
\end{array}
\right)=\left(
\begin{array}{ccc}
1&  e^{t}x&  z  \\
0&     1&e^{-t}y      \\
0& 0&1 
\end{array}
\right)\]
where $H(1)$ is the $3$-dimensional Heisenberg Lie group.
It is known that $G$ has a lattice $\Gamma$ (see \cite{ACF} or \cite{Sa2}).
$G$ admits a left-invariant complex structure $J$ and $(G/\Gamma, J)$ admits a LCK metric but does not admit a Vaisman metric (see \cite{Bel}).
In \cite{Ban} Banyaga gave LCK left-invariant forms $\omega$ and $\omega^{\prime}$ with  Lee forms $\theta$ and $\theta^{\prime}$ respectively such that $[\omega]_{\theta}=0$ and $[\omega^{\prime}]_{\theta^{\prime}}\not=0$.
\end{example}

\begin{theorem}\label{MTTT}
Let $G=\R^{m}\ltimes_{\phi} \R^{n}$  such that $\phi$ is a semi-simple action.
Suppose  $\dim [G,G]>\frac{\dim G}{2}$, $G$ has a lattice $\Gamma$ and a left-invariant complex structure $J$ and $b_{1}(G/\Gamma)=b_{1}(\g)$.
Then $(G/\Gamma,J)$ admits no Vaisman metric. 
\end{theorem}
\begin{proof}
Suppose $(G/\Gamma, J)$ has a Vaisman metric (not necessarily left-invariant) $g$ with the fundamental form $\omega$ and  Lee form $\theta$.
By $b_{1}(G/\Gamma)=b_{1}(\g)$, the inclusion $\bigwedge \g^{\ast}\subset A^{\ast}(G/\Gamma)$ induces an isomorphism of the first cohomology, and so we have a closed invariant $1$-form $\theta_{0}\in \bigwedge \g^{\ast}$ and a function $f$ on $G/\Gamma$ such that $\theta_{0}-\theta=df$.
By the map $A^{\ast}(G/\Gamma)\ni \alpha\mapsto  e^{f}\alpha$,  we have an isomorphism $H^{\ast}_{\theta}(G/\Gamma)\cong H^{\ast}_{\theta_{0}}(G/\Gamma)$.
Consider the invariant form $(e^{f} \omega)_{inv}$ given in section 2.
Then by the definition of  $(e^{f} \omega)_{inv}$, $(e^{f} \omega)_{inv}$ is $J$-invariant and $g_{0}=(e^{f} \omega)_{inv}(-, J-) $ is a positive definite.
By $d_{\theta_{0}}((e^{f} \omega)_{inv} )=I\circ d_{\theta_{0}}(e^{f} \omega )=0$, $g_{0}$ is a left-invariant LCK metric with  Lee form $\theta_{0}$.
Then by the above construction, we have a left-invariant LCK metric on $G/\Gamma$.
By Theorem \ref{LCS},  we have $[(e^{f} \omega)_{inv}]_{\theta_{0}}\not=0$ in $H^{2}_{\theta_{0}}(G/\Gamma)$ and hence $[\omega ]_{\theta}\not=0$ in $H^{2}_{\theta}(G/\Gamma)$.
But this contradicts Theorem \ref{OOR}.
Hence the theorem follows.
\end{proof}
\begin{remark}
For a non-degenerate $2$-form $\omega$, the left-invariant $2$-form $\omega_{inv}$ is not non-degenerate in general.
In this proof, the assumption of the existence of $\omega$-compatible left-invariant almost complex structure is important.
We can rewrite Theorem \ref{LCS}. 
\end{remark}
\begin{theorem}
Let $G=\R^{m}\ltimes_{\phi} \R^{n}$ such that $\phi$ is a semi-simple action.
Suppose  $\dim [G,G]>\frac{\dim G}{2}$ , $G$ has a lattice $\Gamma$ and $b_{1}(G/\Gamma)=b_{1}(\g)$.
Then for a  LCS   form (not necessarily left-invariant)  $\omega$ admitting a left-invariant $\omega$-compatible almost  complex structure, with  Lee form $\theta$, the $d_{\theta}$-cohomology class of $\omega $ is not $0$ in $H_{\theta}^{2}(G/\Gamma)$.
\end{theorem}
\begin{remark}\label{IISO}
By Lemma \ref{sekib}, if $b_{p}(\g)=b_{p}(G/\Gamma)$ we have an isomorphism $H^{p}(\g)\cong H^{p}(G/\Gamma)$.
If $G$ is completely solvable (i.e. for any $g\in G$ the all eigenvalues of the adjoint operator ${\rm Ad}_{g}$ are real numbers), then we have an isomorphism $H^{\ast}(\g)\cong H^{\ast}(G/\Gamma)$ (see \cite{Hatt}).
\end{remark}

\section{On the condition  $H^{1}(\g)\cong H^{1}(G/\Gamma)$}
Let $G$ be a $n$-dimensional simply connected solvable Lie group, $\g$ be the Lie algebra.
and ${\rm Ad}: G\to {\rm Aut}(\g)$ be the adjoint representation.
Denote by ${\rm Ad}_{s g}$  the semi-simple part of ${\rm Ad}_{g}\in {\rm Aut}(\g)$ for $g\in G$.
Since representations of $G$ are triangulizable in $\C$ by Lie's theorem, ${\rm Ad}_{s}:G\to {\rm Aut} (\g_{\C}) $ is a diagonalizable representation.
\begin{definition}\label{wea}
We consider the diagonalization  ${\rm diag}(\alpha_{1}, \dots \alpha_{n})$  of ${\rm Ad}_{s}$.
We call $G$ weakly completely solvable if each $\alpha_{i}$ is not a non-trivial unitary character of $G$.
\end{definition}
Let $\bf T$ be the Zariski-closure of ${\rm Ad}_{s}(G)$ in ${\rm Aut}(\g_{\C})$.
Suppose $G$ has a lattice $\Gamma$.
Let $A_{\Gamma}=\{\alpha^{\prime}_{i}\}$ be the set of characters of $G$ such that for $\alpha^{\prime}_{i}\in A_{\Gamma}$ we can write $\alpha^{\prime}_{i}= \alpha_{i}\circ {\rm Ad}_{s}$ for an algebraic character $\alpha_{i}$ of the algebraic group $\bf T$ and the restriction $\alpha^{\prime}_{i\vert_{\Gamma}}$ is trivial.
\begin{lemma}
$\alpha^{\prime}_{i}\in A_{\Gamma}$ is a unitary character.
\end{lemma}
\begin{proof}
Since the restriction $\alpha^{\prime}_{i\vert_{\Gamma}}$ is trivial, $\alpha^{\prime}_{i}$ induces the function of $G/\Gamma$.
Since $G/\Gamma$ is compact,  the image of $ \alpha^{\prime}_{i}$ is a compact subgroup of $\C^{\ast}$.

\end{proof}
Consider the subDGA 
\[\bigoplus_{\alpha_{i}^{\prime}\in A_{\Gamma}} \left(\alpha^{\prime}_{i}\cdot \bigwedge \g^{\ast}_{\C}\right)\subset A^{\ast}_{\C}(G/\Gamma).\]
Write $\alpha^{\prime}_{i}= \alpha_{i}\circ {\rm Ad}_{s}$ for an algebraic character $\alpha_{i}$ of $\bf T$.
We have the action of $\bf T$ on $\alpha^{\prime}_{i}\cdot \bigwedge \g^{\ast}_{\C}$ given by 
\[t\cdot (\alpha^{\prime}_{i}\omega)=\alpha_{i}(t)^{-1}\alpha^{\prime}_{i}t^{\ast}(\omega).\]
Let
\[\bigoplus_{\alpha_{i}^{\prime}\in A_{\Gamma}} \left(\alpha^{\prime}_{i}\cdot \bigwedge \g^{\ast}_{\C}\right)^{\bf T}\]
be the subDGA which consists of $\bf T$-invariant elements of $\bigoplus_{\alpha_{i}^{\prime}\in A_{\Gamma}} \left(\alpha^{\prime}_{i}\cdot \bigwedge \g^{\ast}_{\C}\right)$.
\begin{theorem}\label{isoc}{\rm ( \cite[Corollary 7.6]{K2})}
The two  inclusions 
\[ \bigoplus_{\alpha_{i}^{\prime}\in A_{\Gamma}} \left(\alpha^{\prime}_{i}\cdot \bigwedge \g^{\ast}_{\C}\right)^{\bf T}\subset \bigoplus_{\alpha_{i}^{\prime}\in A_{\Gamma}} \left(\alpha^{\prime}_{i}\cdot \bigwedge \g^{\ast}_{\C}\right)\subset A^{\ast}_{\C}(G/\Gamma).\]
induce  cohomology isomorphisms.
\end{theorem}

By this theorem, we have:
\begin{theorem}
For a weakly completely solvable Lie group $G$ with a lattice $\Gamma$, we have an isomorphism
\[H^{1}(\g)\cong H^{1}(G/\Gamma).
\]
\end{theorem}
\begin{proof}
For a $1$-form 
\[\omega=\alpha^{\prime}_{1}\omega_{1}+\dots+\alpha^{\prime}_{r}\omega_{r}\in \bigoplus_{\alpha_{i}^{\prime}\in A_{\Gamma}} \left(\alpha^{\prime}_{i}\cdot \bigwedge^{1} \g^{\ast}_{\C}\right)^{\bf T},
\]
we have
\[\omega=({\rm Ad}_{sg})^{\ast}\omega
=\alpha^{\prime}_{1}(g)^{-1}\alpha^{\prime}_{1}{\rm Ad}_{sg}^{\ast}(\omega_{1})+\dots+\alpha^{\prime}_{r}(g)^{-1}\alpha^{\prime}_{r}{\rm Ad}_{sg}^{\ast}(\omega_{r}).
\]
Thus we have ${\rm Ad}_{sg}^{\ast}(\omega_{i})=\alpha^{\prime}_{i}(g)\omega_{i}$ and so
the unitary characters $\alpha_{i}^{\prime}$ are eigencharacters of the diagonalizable representation ${\rm Ad}^{\ast}_{s}$.
But by the condition of a weakly completely solvable Lie group, $\alpha_{i}^{\prime}$ is trivial character and hence we have 
\[\bigoplus_{\alpha_{i}^{\prime}\in A_{\Gamma}} \left(\alpha^{\prime}_{i}\cdot \bigwedge \g^{\ast}_{\C}\right)^{\bf T}
\subset \bigwedge \g^{\ast}_{\C}.\]
By Theorem \ref{isoc}, we have
\[\dim_{\C} H^{1}(G/\Gamma,\C)=\dim_{\C} H^{1}(\g_{\C}).
\]
\end{proof}
Thus we have:
\begin{corollary}\label{WCS}
Let $G=\R^{m}\ltimes_{\phi} \R^{n}$  such that $\phi$ is a semi-simple action.
Suppose $G$ is weakly completely solvable,  $\dim [G,G]>\frac{\dim G}{2}$ and $G$ has a lattice $\Gamma$ and a left-invariant complex structure $J$.
Then $(G/\Gamma,J)$ admits no Vaisman metric. 
\end{corollary}

\begin{remark}
We call a Lie group $G$ of exponential type if the exponential map $\exp :\g\to G$ is a diffeomorphism.
A simply connected solvable Lie group $G$ is of exponential type if and only if for any $g\in G$ ${\rm Ad}_{sg}$ has no unitary eigenvalue $\not=1$ (see \cite{Aus}).
Thus the class of weakly solvable Lie groups contains the class of  solvable Lie groups  of exponential type.
Since a non-trivial holomorphic character of a complex solvable Lie group is not unitary,
the  class of weakly solvable Lie groups contains the class of complex solvable Lie groups.
Since for a complex solvable Lie group $G$, the abelianization $G/[G,G]$ is also complex and hence $b_{1}(G/\Gamma)=b_{1}(\g)$ is even.
Since the first betti number of a compact Vaisman manifold is odd,
 a compact complex parallelizable solvmanifold $G/\Gamma$ admits no complex structure admitting a compatible Vaisman metric.
\end{remark}
\begin{remark}
For a weakly completely solvable Lie group $G$ with a lattice $\Gamma$,  an isomorphism
\[H^{p}(\g)\cong H^{p}(G/\Gamma)
\]
does not hold  for $2\le p$ in general.
For example,  we consider the complex solvable Lie group $G=\C\ltimes_{\phi} \C^{2}$ with $\phi(x)={\rm diag}(e^{x},e^{-x})$.
By the above theorem, for any lattice $\Gamma$ we have an isomorphism $H^{1}(\g)\cong H^{1}(G/\Gamma)$.
But for some lattice $\Gamma$, an isomorphism $H^{2}(\g)\cong H^{2}(G/\Gamma)$ does not hold (see \cite{BT}).
\end{remark}
For a simply connected solvable Lie group $G$ with a left-invariant complex structure $J$   and a lattice $\Gamma$ satisfying $ H^{1}(G/\Gamma)\cong  H^{1}(\g)$,
in the proof of Theorem \ref{MTTT}, we observe that  if $G/\Gamma$ admits a non-left-invariant LCK, then
 we can make a left-invariant LCK metric $(e^{f} \omega)_{inv}(-, J-)$.
Hence we have: 
\begin{corollary}
For a weakly completely solvabale Lie group $G$ with a left-invariant complex structure $J$ and a lattice $\Gamma$, $(G/\Gamma,J)$ admits a LCK metric if and only if $(G/\Gamma,J)$ admits a left-invariant LCK metric.
\end{corollary}

\section{Oeljeklaus-Toma manifolds  as solvmanifolds}
In this section, we give solvmanifold-presentations $G/\Gamma$ of examples given by Oeljeklaus and Toma in \cite{OT}.
By such presentations, we will show the non-existence of Vaisman metrics and give some remarks.

Let $K$ be a finite extension field of $\Q$ of degree $s+2t$ ($s>0$, $t>0$).
Suppose $K$ admits embeddings $\sigma_{1},\dots \sigma_{s},\sigma_{s+1},\dots, \sigma_{s+2t}$ into $\C$ such that $\sigma_{1},\dots ,\sigma_{s}$ are real embeddings and $\sigma_{s+1},\dots, \sigma_{s+2t}$ are complex ones satisfying $\sigma_{s+i}=\bar \sigma_{s+i+t}$ for $1\le i\le t$. 
For any $s$ and $t$, we can choose $K$ admitting such embeddings (see \cite{OT}).
Let ${\mathcal O}_{K}$ be the ring of algebraic integers of $K$, ${\mathcal O}_{K}^{\ast}$ the group of units in ${\mathcal O}_{K}$ and 
\[{\mathcal O}_{K}^{\ast\, +}=\{a\in {\mathcal O}_{K}^{\ast}: \sigma_{i}(a)>0 \,\, {\rm for \,\,  all}\,\, 1\le i\le s\}.
\]  
Define $\sigma :{\mathcal O}_{K}\to \R^{s}\times \C^{t}$ by
\[\sigma(a)=(\sigma_{1}(a),\dots ,\sigma_{s}(a),\sigma_{s+1}(a),\dots ,\sigma_{s+t}(a))
\]
for $a\in {\mathcal O}_{K}$.
Define $l:{\mathcal O}_{K}^{\ast\, +}\to \R^{s+1}$ by 
\[l(a)=(\log \vert \sigma_{1}(a)\vert,\dots ,\log \vert \sigma_{s}(a)\vert , 2\log \vert \sigma_{s+1}(a)\vert,\dots ,2\log \vert \sigma_{s+t}(a)\vert)
\]
for $a\in {\mathcal O}_{K}^{\ast\, +}$.
Then by Dirichlet's units theorem, $l({\mathcal O}_{K}^{\ast\, +})$ is a lattice in the vector space $L=\{x\in \R^{s+t}\vert \sum_{i=1}^{s+t} x_{i}=0\}$.
Consider the projection $p:L\to \R^{s}$ given by the first $s$ coordinate functions.
Then we have a  subgroup $U$ with the rank $s$ of ${\mathcal O}_{K}^{\ast\, +}$ such that $p(l(U))$ is a lattice in $\R^{s}$.
Write $l(U)=\Z v_{1}\oplus\dots\oplus \Z v_{s}$ for generators $v_{1},\dots v_{s}$ of $l(U)$.
For the standerd basis $e_{1},\dots ,e_{s+t}$ of $\R^{s+t}$, we have
\[\sum_{j=1}^{s} a_{ij}v_{j}=e_{i}+\sum_{k=1}^{t}b_{ik}e_{s+k}
\]
for any $1\le i\le s$.
Consider the complex half plane $H=\{z\in \C: {\rm Im} z>0\}=\R\times \R_{>0}$.
We have the action of $U\ltimes{\mathcal O}_{K}$ on $H^{s}\times \C^{t}$
such that 
\begin{multline*}
(a,b)\cdot (x_{1}+\sqrt{-1}y_{1},\dots ,x_{s}+\sqrt{-1}y_{s}, z_{1},\dots ,z_{t})\\
=(\sigma_{1}(a)x_{1}+\sigma_{1}(b)+\sqrt{-1} \sigma_{1}(a)y_{1}, \dots ,\sigma_{s}(a)x_{s}+\sigma_{s}(b)+\sqrt{-1} \sigma_{s}(a)y_{s},\\
 \sigma_{s+1}(a)z_{1}+\sigma_{s+1}(b),\dots ,\sigma_{s+t}(a)z_{t}+\sigma_{s+t}(b)).
\end{multline*}
In \cite{OT} it is proved that the quotient $H^{s}\times \C^{t}/U\ltimes{\mathcal O}_{K}$ is compact.
We call this complex manifold a  Oeljeklaus-Toma (OT) manifold of type $(s,t)$.

For $a \in U$ and $(t_{1},\dots ,t_{s})=p(l(a))\in p(l(U))$, since $l(U)$ is generated by the basis $v_{1},\dots,v_{s}$ as above, $l(a)$ is a linear combination of $e_{1}+\sum_{k=1}^{t}b_{1k}e_{s+k},\dots ,e_{s}+\sum_{k=1}^{t}b_{sk}e_{s+k}$ and hence we have 
\[l(a)=\sum_{i=1}^{s}t_{i}(e_{i}+\sum_{k=1}^{t}b_{ik}e_{s+k})=(t_{1},\dots ,t_{s}, \sum_{i=1}^{s}b_{i1}t_{i},\dots,\sum_{i=1}^{s}b_{it}t_{i}).
\]
By $2\log \vert \sigma_{s+k}(a)\vert=\sum_{i=1}^{s}b_{ik}t_{i}$, we can write 
\[\sigma_{s+k}(a)=e^{\frac{1}{2}\sum_{i=1}^{s}b_{ik}t_{i}+\sqrt{-1}\sum_{i=1}^{s}c_{ik}t_{i}}\]
for some $c_{ik}\in \R$.
We consider the Lie group $G=\R^{s}\ltimes_{\phi} (\R^{s}\times \C^{t})$ with
\[
\phi(t_{1},\dots ,t_{s})\\
={\rm diag}(e^{t_{1}},\dots ,e^{t_{s}},e^{\psi_{1}+\sqrt{-1}\varphi_{1}},\dots ,e^{\psi_{t}+\sqrt{-1}\varphi_{t}})\]
 where $\psi_{k}=\frac{1}{2}\sum_{i=1}^{s}b_{ik}t_{i}$ and $\varphi_{k}=\sum_{i=1}^{s}c_{ik}t_{i}$.
Then for $(t_{1},\dots ,t_{s})\in p(l(U))$, we have 
\[\phi(t_{1},\dots ,t_{s})(\sigma ({\mathcal O}_{K}))\subset \sigma({\mathcal O}_{K}).\]
By the embedding 
\[U\ltimes{\mathcal O}_{K} \ni (a,b)\mapsto (p(l(a)), \sigma(b))\in G,\]
the group $U\ltimes{\mathcal O}_{K} $ is a subgroup of $G$.
Since $p(l(U))$ and $\sigma({\mathcal O}_{K})$ are lattices in $\R^{s}$ and $\R^{s}\times \C^{t}$ respectively, the subgroup $U\ltimes{\mathcal O}_{K} $ is a lattice in $G$.
By the correspondence 
\begin{multline*}
H^{s}\times \C^{t}\ni (x_{1}+\sqrt{-1}y_{1},\dots ,x_{s}+\sqrt{-1}y_{s},z_{1},\dots,z_{t})\\
\mapsto (x_{1},\log y_{1},\dots ,x_{s}, \log y_{s},z_{1},\dots,z_{t})\in \R^{s}\times \R^{s}\times \C^{t},
\end{multline*}
we can identify the action of $U\ltimes{\mathcal O}_{K} $ on $H^{s}\times \C^{t}$  with the left action of the lattice $U\ltimes{\mathcal O}_{K} $ on $G$.
Hence OT-manifold $H^{s}\times \C^{t}/U\ltimes{\mathcal O}_{K}$ is considered as a solvmanifold $G/U\ltimes{\mathcal O}_{K}$.
Let $\g$ be the Lie algebra of $G$.
Then $\bigwedge \g^{\ast}$ is generated by $\{\alpha_{1}, \dots ,\alpha_{s}, \beta_{1}, \dots, \beta_{s}, \gamma_{1},\gamma_{2},\dots ,\gamma_{2t-1}, \gamma_{2t}\}$ such that the differential is given by
\[d\alpha_{i}=0,\  d\beta=-\alpha_{i} \wedge \beta_{i},
\]
\[d\gamma_{2i-1}=\bar\psi_{i}\wedge  \gamma_{2i-1}+\bar\varphi_{i} \wedge \gamma_{2i},\,  d\gamma_{2i}=-\bar\varphi _{i}\wedge\gamma_{2i-1}+  \bar\psi_{i}\wedge \gamma_{2i},
\]
where $\bar\psi_{i}=\frac{1}{2}\sum_{i=1}^{s}b_{ik}\alpha_{i}$ and $\bar\varphi_{i} =\sum_{i=1}^{s}c_{ik}\alpha_{i}$.
Consider $w_{i}=\alpha_{i}+\sqrt{-1}\beta_{i}$ for $1\le i\le s$ and  $w_{s+i}=\gamma_{2i-1}+\sqrt{-1} \gamma_{2i}$ as $(1,0)$-forms.
Then $w_{1},\dots ,w_{s+t}$ gives a left-invariant almost complex structure $J$.
By the computations of $dw_{i}$, $J$ is integrable.
The complex manifold $(G/U\ltimes{\mathcal O}_{K},J)$ is a presentation of OT-manifold $H^{s}\times \C^{t}/U\ltimes{\mathcal O}_{K}$ as a solvmanifold.
\begin{remark}
For $s=1, t=1$, we have $G=\R\ltimes_{\phi} (\R\times \C)$ with $\phi(t)={\rm diag}(e^{t},e^{-\frac{1}{2}t+\sqrt{-1}ct})$.
It is known that this $G$ admits a lattice $\Gamma$ and $G/\Gamma$ is the  Inoue surface $S^{0}$ (see \cite{HK}).
\end{remark}
\begin{remark}
We can not take $\varphi_{k}=0$.
In fact, the completely solvable Lie group $\R\ltimes_{\phi^{\prime}} \R^{3}$ with $\phi^{\prime}(t)={\rm diag}(e^{t},e^{-\frac{1}{2}t},e^{-\frac{1}{2}t})$ admits no lattice (see \cite{HK}).
\end{remark}
\begin{remark}
$J$ is left-invariant but not right-invariant and so $(G,J)$ is not a complex Lie group.
In fact,  $J$ is not fixed by the right action of any non-trivial element of $G$.
Hence the right action of $G$ on $(G/U\ltimes{\mathcal O}_{K},J)$ is not holomorphic.
In \cite{OT}, it is proved that the group of holomorphic automorphisms of each OT-manifold is discrete.
\end{remark}

 We have:
\begin{corollary}
For any lattice $\Gamma$ of $G$,  $(G/\Gamma, J)$ admits no Vaisman metric.
\end{corollary}
\begin{proof}
Since we have $H^{\ast}(\g)=\langle [\alpha_{1}],\dots ,[\alpha_{s}]\rangle$, we have
\[\dim [G,G]=2t+2s-s=2t+s>t+s=\frac{\dim G}{2}.\]
Since $G$ is weakly completely solvable, the corollary follows from Corollary \ref{WCS}.
\end{proof}
Hence we have:
\begin{corollary}\label{vaio}
 OT-manifolds do not   admit  Vaisman metrics.
\end{corollary}

We consider the case $t=1$.
We can take $U={\mathcal O}_{K}^{\ast\, +}$ and any $U$ is a finite index subgroup of ${\mathcal O}_{K}^{\ast\, +}$.
For $U={\mathcal O}_{K}^{\ast\, +}$,  we have  
\[(\sigma_{1}(a), \dots \sigma_{s}(a), \sigma_{s+1}(a))=(e^{t_{1}}, \dots, e^{t_{s}}, e^{-\frac{1}{2}(t_{1}+\dots +t_{s})}e^{\sqrt{-1}\varphi_{1}}),
\]
for  $l(a)=(t_{1}, \dots, t_{s},-t_{1}-\dots -t_{s} )\in l({\mathcal O}_{K}^{\ast\, +}) \subset L$.
Hence we have  $G=\R^{s}\ltimes _{\phi}(\R^{s}\times \C) $ such that 
\[\phi(t_{1},\dots ,t_{s})={\rm diag}(e^{t_{1}},\dots ,e^{t_{s}},e^{-\frac{1}{2}(t_{1}+\dots+ t_{s})+\sqrt{-1}\varphi_{1}}).
\]
Then $\bigwedge \g^{\ast}$ is generated by $\{\alpha_{1}, \dots ,\alpha_{s}, \beta_{1}, \dots , \beta_{s}, \gamma_{1},\gamma_{2}\}$ such that the differential is given by
\[d\alpha_{i}=0,\  d\beta=-\alpha_{i} \wedge \beta_{i},
\]
\[d\gamma_{1}=\frac{1}{2}\theta\wedge \gamma_{1}+\bar\varphi_{1} \wedge \gamma_{2},\,  d\gamma_{2}=-\bar\varphi_{1} \wedge\gamma_{1}+  \frac{1}{2}\theta\wedge \gamma_{2},
\]
where $\theta=\alpha_{1}+\dots \alpha_{s}$.
Consider $w_{i}=\alpha_{i}+\sqrt{-1}\beta_{i}$ for $1\le i\le s$ and  $w_{s+1}=\gamma_{1}+\sqrt{-1} \gamma_{2}$ as $(1,0)$-forms.
Then $w_{1},\dots ,w_{s+1}$ gives a left-invariant  complex structure $J$.
Consider 
\[\omega=\sum^{n}_{i}2\alpha_{i}\wedge\beta_{i}+\sum_{i\not=j}\alpha_{i}\wedge \beta_{j}+\gamma_{1}\wedge \gamma_{2}.\]
Since we have 
\begin{multline*}
\sum_{\substack{i,j,k \\ {i\not=j}}}\alpha_{k}\wedge\alpha_{i}\wedge\beta_{j}
=\sum_{i\not=j}\alpha_{j}\wedge\alpha_{i}\wedge\beta_{j}+\sum_{\substack{i,j,k \\
k\not=i,k\not=j, i\not=j}}\alpha_{k}\wedge\alpha_{i}\wedge\beta_{j}\\
=
\sum_{i\not=j}\alpha_{j}\wedge\alpha_{i}\wedge\beta_{j}+\sum_{k<i}\alpha_{k}\wedge\alpha_{i}\wedge\beta_{j}
-\sum_{i<k}\alpha_{i}\wedge\alpha_{k}\wedge\beta_{j}\\
=\sum_{i\not=j}\alpha_{j}\wedge\alpha_{i}\wedge\beta_{j},
\end{multline*}
we have 
\begin{multline*}
\theta\wedge\omega
=\sum_{i,k} 2\alpha_{k}\wedge\alpha_{i}\wedge\beta_{i}+
\sum_{\substack{i,j,k \\ {i\not=j}}}\alpha_{k}\wedge\alpha_{i}\wedge\beta_{j}+\theta\wedge\gamma_{1}\wedge\gamma_{2}\\
=\sum_{i,j}2\alpha_{i}\wedge\alpha_{j}\wedge\beta_{j}+
\sum_{i\not=j}\alpha_{j}\wedge\alpha_{i}\wedge\beta_{j}+\theta\wedge\gamma_{1}\wedge\gamma_{2}\\
=\sum_{i\not=j}\alpha_{i}\wedge\alpha_{j}\wedge \beta_{j}+\theta\wedge\gamma_{1}\wedge\gamma_{2},
\end{multline*}
and so we have $d\omega=\theta\wedge \omega$.
Thus for $g=\omega(-,J-)$ $(g,J)$ is a left-invariant LCK metric on $G$.
\begin{remark}
In \cite{OT}, Oeljeklaus and Toma gave a K\"ahler potential on  $H^{s}\times \C$ which gives a LCK structure on $H^{s}\times \C/{\mathcal O}_{K}^{\ast\, +}\ltimes{\mathcal O}_{K}$.
As above, in this paper, we have represented such LCK metric as a left-invariant form on $G$.
\end{remark}

\begin{remark}
Studying the action of ${\mathcal O}_{K}^{\ast\, +}\ltimes{\mathcal O}_{K}$ on the K\"ahler potential constructed in \cite{OT}, in \cite{PV} Parton and Vuletescu compute the rank  of this LCK metric on each  OT-manifold.
Here the rank is an invariant of a conformal class of LCK metric (see \cite{PV} for definition).
By this computation, they showed that the rank of this LCK metric is equal to $s$ or $\frac{s}{2}$.
This result shows how much  this LCK metric is  different from a Vaisman metric, because the rank of Vaisman metric is equal to $1$ (see \cite{GOP}).
\end{remark}
In this paper, we proved the non-existence of Vaisman metrics on OT-manifolds.
By Corollary \ref{vaio}, we have:
\begin{corollary}
For any $s>0$,  OT-manifolds of type $(s,1)$ are LCK manifolds not admitting Vaisman metrics.
\end{corollary}

To give another important remark concerning a Vaisman metric, we consider $s=2$.
\begin{proposition}{\rm (\cite{OT})}
For an OT-manifold of type $(2,1)$, we have $b_{1}=b_{5}=2$, $b_{2}=b_{4}=1$ and $b_{3}=0$.
\end{proposition}
In this case the LCK metric is given by
\[g=2\alpha_{1}^{2}+2\beta^{2}_{1}+2\alpha_{1}\cdot \alpha_{2}-2\beta_{1} \cdot \beta_{2}+2\alpha_{2}^{2}+ 2\beta_{2}^{2}+\gamma_{1}^2+\gamma_{2}^{2}.
\]
We call a Riemaniann metric formal if all products of harmonic forms are again harmonic (see \cite{Kot}).
\begin{proposition}
$g$ is a formal metric on an OT-manifold of type $(2,1)$.
\end{proposition}
\begin{proof}
For the metric $g$, the left invariant forms $\alpha_{1}$, $\alpha_{2}$, $\alpha_{1}\wedge\alpha_{2}$, $\beta_{1}\wedge \beta_{2}\wedge \gamma_{1}\wedge \gamma_{2}$, $\alpha_{1}\wedge \beta_{1}\wedge \beta_{2}\wedge \gamma_{1}\wedge \gamma_{2}$, $\alpha_{2}\wedge \beta_{1}\wedge \beta_{2}\wedge \gamma_{1}\wedge \gamma_{2}$ and $\alpha_{1}\wedge\alpha_{2}\wedge\beta_{1}\wedge \beta_{2}\wedge \gamma_{1}\wedge \gamma_{2}$ are harmonic forms.
By the Betti numbers of an OT-manifold of type $(2,1)$, the space of the all harmonic forms is spanned by these forms.
Hence all products of harmonic forms are again harmonic.
\end{proof}
\begin{remark}
In \cite{OP}, it is proved that a Vaisman metric on compact manifold $M$ is a formal metric if and only if $b_{1}(M)=b_{2n+1}(M)=1$ and $b_{k}(M)=0$ for $2\le k\le 2n$. 
On the other hand, for a general LCK metric on compact manifold $M$, Ornea and Pilca's theorem does not hold.
\end{remark}
\begin{remark}
The following problems remain.
\end{remark}
\begin{problem}
In $G$, does there exist a lattice which can not be constructed by Oeljeklaus and Toma's technique?
\end{problem}
\begin{problem}
For odd $s>0$ does an OT-manifold of type $(s,1)$ admit  a non-invariant complex structure admitting a compatible Vaisman metric?
\end{problem}
{\bf  Acknowledgements.} 

The author would like to express his gratitude to   Toshitake Kohno for helpful suggestions and stimulating discussions.
He would also like to thank  Keizo Hasegawa  for valuable comments.
This research is supported by JSPS Research Fellowships for Young Scientists.


\begin{thebibliography}{40}
\bibitem{ACF}
L. C. de Andr\'es,  L. A. Cordero,  M. Fern\'andez, J. J.  Menc\'ia,  Examples of four-dimensional compact locally conformal K\"ahler solvmanifolds. Geom. Dedicata {\bf 29} (1989), no. 2, 227--232.
\bibitem{Aus}
L. Auslander, An exposition of the structure of solvmanifolds. I. Algebraic theory.  Bull. Amer. Math. Soc. {\bf 79}  (1973), no. 2, 227--261.
\bibitem{Ban} A. Banyaga,
Examples of non $d_{\omega}$-exact locally conformal symplectic forms. J. Geom. {\bf 87} (2007), no. 1-2, 1--13.
\bibitem{BT} P. de Bartolomeis, A. Tomassini, On solvable generalized Calabi-Yau manifolds. Ann. Inst. Fourier (Grenoble) {\bf 56} (2006), no. 5, 1281--1296.
\bibitem{BC} O. Baues,  V. Cort\'es,
Aspherical K\"ahler manifolds with solvable fundamental group. Geom. Dedicata {\bf 122} (2006), 215--229.
\bibitem{Bel}
F. A. Belgun,
On the metric structure of non-K\"ahler complex surfaces. Math. Ann. {\bf 317} (2000), no. 1, 1--40.
\bibitem{Bor}
A. Borel, Linear algebraic groups 2nd enl. ed Springer-verlag (1991).
\bibitem{Dix}
J. Dixmier, 
Cohomologie des alg\'ebres de Lie nilpotentes. 
Acta Sci. Math. Szeged {\bf 16} (1955), 246--250.
\bibitem{DO} S.  Dragomir, L. Ornea,  Locally conformal K\"ahler geometry. Progress in Mathematics, 155. Birkh\"auser Boston, Inc., Boston, MA, 1998.
\bibitem{GOP}
R. Gini, L. Ornea, M.  Parton, P. Piccinni,  Reduction of Vaisman structures in complex and quaternionic geometry. J. Geom. Phys. {\bf 56} (2006), no. 12, 2501--2522.
\bibitem{FT}
A. Fino, A. Tomassini, Non-K\"ahler solvmanifolds with generalized K\"ahler structure. J. Symplectic Geom. {\bf 7} (2009), no. 2, 1--14.
\bibitem{Hn} K. Hasegawa, A note on compact solvmanifolds with K\"ahler structures. Osaka J. Math. {\bf 43} (2006), no. 1, 131--135.
\bibitem{HK} K. Hasegawa, Y. Kamishim,
Locally conformal Kaehler structures on homogeneous spaces. http://arxiv.org/abs/1101.3693.
\bibitem{Hatt} A. Hattori, Spectral sequence in the de Rham cohomology of fibre bundles. J. Fac. Sci. Univ. Tokyo Sect. I {\bf 8} 1960 289--331 (1960). 
\bibitem{Kas} T. Kashiwada, S. Sato, On harmonic forms in compact locally conformal K\"ahler manifolds with the parallel Lee form. Ann. Fac. Sci. Univ. Nat. Za\"ire (Kinshasa) Sect. Math.-Phys. {\bf 6} (1980), no. 1-2, 17--29.
\bibitem{K1} H. Kasuya, Formality and hard Lefschetz properties of aspherical manifolds. to appear in Osaka J. Math. http://arxiv.org/abs/0910.1175.
\bibitem{KS} H. Kasuya, Cohomologically symplectic solvmanifolds are symplectic.  J. Symplectic  Geom. {\bf 9} (2011) no. 4, 429--434.
\bibitem{K2} H. Kasuya, Minimal models, formality and hard Lefschetz properties of solvmanifolds with local systems.
http://arxiv.org/abs/1009.1940.
\bibitem{KF} H. Kasuya, Geometrical formality of solvmanifolds and solvable Lie type geometries. in preparation.
\bibitem{Kot}
D. Kotschick,
On products of harmonic forms. Duke Math. J.  {\bf 107} (2001), no. 3, 521--531.
\bibitem{LL}
M. de Le\'on,  B. L\'opez,  J. C.  Marrero,  E. Padr\'on,  On the computation of the Lichnerowicz-Jacobi cohomology. J. Geom. Phys. {\bf 44} (2003), no. 4, 507--522. 


\bibitem{OT}K. Oeljeklaus, M. Toma,  Non-K\"ahler compact complex manifolds associated to number fields. Ann. Inst. Fourier (Grenoble) {\bf 55} (2005), no. 1, 161--171.
\bibitem{OP} L. Ornea, M. Pilca, Remarks on the product of harmonic forms,  {\bf 250} (2011), Pacific J. Math. No. 2, 353--363.
\bibitem{OV} L. Ornea, M. Verbitsky,  Morse-Novikov cohomology of locally 
conformally K\"ahler manifolds. J. Geom. Phys. {\bf 59} (2009), no. 3, 295--305.
 \bibitem{PV} M. Parton, V. Vuletescu, Examples of non-trivial rank in locally conformal K\"ahler geometry, Math. Z. (2010), DOI 10.1007/s00209-010-0791-5.
 
 \bibitem{R}
 M.S. Raghunathan, Discrete subgroups of Lie Groups, Springer-verlag, New York, 1972.
 \bibitem{Saw} 
 H. Sawai, Locally conformal K\"ahler structures on compact nilmanifolds with left-invariant complex structures. Geom. Dedicata {\bf 125} (2007), 93--101.
 \bibitem{Sa2} H. Sawai, A construction of lattices on certain solvable Lie groups. Topology Appl. {\bf 154} (2007), no. 18, 3125--3134.
 \bibitem{V}
I. Vaisman,  Generalized Hopf manifolds. Geom. Dedicata {\bf 13} (1982), no. 3, 231--255.
\bibitem{Y}
T. Yamada, A pseudo-K\"ahler structure on a nontoral compact complex parallelizable solvmanifold. Geom. Dedicata {\bf 112} (2005), 115--122.

\end{thebibliography}
\end{document}